\documentclass[11pt]{article}
\usepackage{times}
\usepackage{amsmath,amssymb}
\usepackage{theoremref,hyperref}
\usepackage[utf8]{inputenc}
\usepackage{indentfirst}
\usepackage{graphicx,tikz}
\usepackage{xcolor}
\usepackage{xskak}
\usepackage[T1]{fontenc}

\usepackage{amsthm}
\newtheorem{problem}{Question}
\newtheorem{definition}[problem]{Definition}
\newtheorem{theorem}[problem]{Theorem}

\newtheorem{lemma}[problem]{Lemma}
\newtheorem{conjecture}[problem]{Conjecture}

\title{Torus Queen Independence}
\author{Kada K Williams}
\date{April 28, 2024}

\begin{document}

\maketitle

\begin{abstract}
    Define a queen on $\mathbb{Z}_n^d$ with admissible moves parallel to $\mathbf{x}\in\{-1,0,1\}^d$ at arbitrary length. How many queens can be placed on $\mathbb{Z}_n^d$ without any two in conflict? In 1918, P\'{o}lya showed that $n$ queens on $\mathbb{Z}_n^2$ cannot be conflict-free if $n$ is a multiple of $2$ or $3$. We conjecture and partially prove that no $n^{d-1}$ queens on $\mathbb{Z}_n^d$ are conflict-free if $n$ is divisible by a prime less than $2^d$. On the other hand, extending work by Monsky from 1989, we show there are $(n-c)n^{d-2}$ conflict-free queens on $\mathbb{Z}_n^d$ for some constant $c=c(d)$.
\end{abstract}

\section{Introduction}

According to Hedetniemi and Hedetniemi \cite{HeH}, significant attention has been devoted to a variety of problems that involve chess pieces upon an array of size $n\times n$. Each figurine, such as a bishop, a rook, or a queen, has a prescribed set of moves. If one figurine could move to a specific field occupied by another figurine, the two are in conflict. In this paper, we study at most how many of a certain figurine can be placed on a board without any conflict.

In the case of rooks, which can move to the same row or column, the maximum is $n$. On the $n\times n$ board, there is sufficient wiggle room to place $n$ queens \cite{HeH}, but not on the board $\mathbb{Z}_n^2$, where we permit moving diagonally across the boundary to the other side, if $n$ is a multiple of $2$ or $3$ \cite{Pol}. Note that $\mathbb{Z}_n$ is shorthand for $\mathbb{Z}/n\mathbb{Z}$.

In the first version of this paper, we showed the maximum number of conflict-free queens on the torus $\mathbb{Z}_n^2$ is $n$ if $n\equiv \pm1\pmod{6}$, $n-1$ if $n\equiv \pm2\pmod{12}$, and conjectured it is $n-2$ otherwise. Although our examples were of $n-4$ or $n-3$ queens, it turned out that Monsky had given constructions of $n-2$ conflict-free queens if $n\equiv 0,\pm3,\pm4,6\pmod{12}$ \cite{Mon}, which the author then verified \cite{Wil}. This having been resolved elsewhere, we omit the partial solution from this version.

A queen in higher dimensions, on $\mathbb{Z}_n^d$, could move along a line parallel to a nonzero tuple in $\{-1,0,1\}^d$ \cite{Kla}, \cite{VaR} or along a hyperplane whose normal is such a tuple \cite{Nud}. In the latter case, where a queen at $\mathbf{x}\in \mathbb{Z}_n^d$ can move to any field where the sum of two disjoint sets of components is invariant, by considering the first component, at most $n$ queens are possible. This can be attained if $n$ is coprime with all primes less than $2^d$, but not if $(n,(2d)!)>1$ \cite{Nud}. One may conjecture that such a condition also holds for line queens: $n^{d-1}$ conflict-free queens if and only if $(n,2^d!)=1$ \cite{BeS}. Pólya proved this when $d=2$ \cite{Pol} and Van Rees conjectured this when $d=3$ \cite{VaR}. We justify the conjecture by proving it for a set of $n$ that includes all squarefree numbers. The first step of our proof is the key idea in proving Proposition 11 of Burger, Cockayne, and Mynhardt \cite{BCM}; multinomial sum proof, à la Nudelman \cite{Nud}, only seems to follow through if $(n,5^2)=5$ \cite{Wil}.

Finally, let $Q_d(n)$ be the maximum number of conflict-free queens on $\mathbb{Z}_n^d$. With an explicit example, we prove that $Q_d(n)\ge n^{d-1}-c(d)n^{d-2}$, as conjectured by the author \cite{Wil}.

\section{Sharp upper bounds in $\mathbb{Z}_n^2$}

In 1918, Pólya \cite{Pol} proved that if $n$ is not coprime to $2\cdot 3$, then it is impossible to place $n$ queens on $\mathbb{Z}_n^2$. The reasoning involved can be sharpened as follows.

\begin{lemma}
Let $n$ be an odd positive integer, and suppose that $n-1$ queens are placed on $\mathbb{Z}_n^2$ without conflict. Then one can place one more queen without conflict.
\end{lemma}

\begin{proof}
Let our queens occupy the fields $(x_1,y_1),\dots,(x_{n-1},y_{n-1})$. These are conflict-free if and only if the values of $x_i$ are distinct, and similarly of $y_i$, $x_i+y_i$, and $x_i-y_i$. Our objective is to find which value in $\mathbb{Z}_n$ is not attained.

By translating the rows and columns, we may suppose that $x_i$ and $y_i$ take on every value but $0$. Let us then inquire as for whether a queen can be placed on $(0,0)$ without conflict. 

Denote by $a_+$ and $a_-$ the missing value of $x_i+y_i$ and $x_i-y_i$, respectively. If we can determine these to equal $0$, then the diagonals of $(0,0)$ are not in check.

It is known that $\sum_{i=1}^{n-1}x_i$ and $\sum_{i=1}^{n-1}y_i$ equal $S_1=\sum_{i=1}^{n-1}i=\frac{n(n-1)}{2}$. If $n$ is odd, this is divisible by $n$. Therefore,
$$\sum_{i=1}^{n-1}(x_i\pm y_i)\equiv S_1-a_\pm\pmod n$$
yields $a_\pm \equiv 0\pmod n$, and a queen on $(0,0)$ is established without conflict.
\end{proof}

Since placing $n-1$ conflict-free queens can determine the missing column, row, difference diagonal, and sum diagonal, the congruences in \cite{Pol} can be improved.

\begin{theorem}
Let $n$ be a multiple of $3$ or a multiple of $4$. Then it is impossible to place $n-1$ queens on $\mathbb{Z}_n^2$ without conflict.
\end{theorem}

\begin{proof}
Let our queens occupy the fields $(x_1,y_1),\dots,(x_{n-1},y_{n-1})$.

As before, we suppose that $x_i$ and $y_i$ do not attain the value $0$, and the missing value for $x_i\pm y_i$ is $a_\pm\equiv \sum_{i=1}^{n-1}i\pmod n$. If $n$ is odd, this is congruent to $0$, while if $n$ is even, this is congruent to $\frac n2$ modulo $n$. The sum of $x_i^2$ or of $y_i^2$ is 
$$S=\sum_{i=1}^{n-1}i^2 =\frac{(n-1)n(2n-1)}{6}.$$ 
The sum of $(x_i\pm y_i)^2$ is $\sum_{i=1}^{n-1}(x_i\pm y_i)^2\equiv S-a_\pm^2\pmod n$.

Next, we shall make use of the parallelogram identity
$$\sum_{i=1}^{n-1}(x_i+y_i)^2+\sum_{i=1}^{n-1}(x_i-y_i)^2=2\sum_{i=1}^{n-1}x_i^2+2\sum_{i=1}^{n-1}y_i^2.$$
Plugging in, we obtain $2\left(S-a_\pm^2\right)\equiv 4S\pmod n$, so $2S\equiv -2a_\pm^2\pmod n$. If $n$ is an odd multiple of $3$, we have a contradiction, given $a_\pm\equiv 0\pmod n$ and
$$2S=\frac{(n-1)n(2n-1)}{3}=\frac{n}{3}\cdot (n-1)(2n-1)\equiv \frac n3\pmod n.$$

If $n$ is even, a stronger congruence holds, because a shift of $X$ by $+n$ will increment $X^2$ by $(X+n)^2-X^2=2nX+n^2$, a multiple of $2n$. Hence, the sum of squares congruence actually holds modulo $2n$, yielding $S\equiv -a_\pm^2\pmod {n}$.

If $n$ is divisible by $4$, then $a_\pm^2=\frac{n^2}{4}$ vanishes modulo $n$. However, $S$ is not divisible by $n$, because $\frac{(n-1)(2n-1)}{6}$ is not an integer. This is a contradiction.

Since either case is contradictory, a placement of $n-1$ queens is impossible.
\end{proof}

\section{Upper bounds in higher dimensions}

\begin{definition}
A queen on $\mathbb{Z}_n^d$ can move by $\lambda\in\mathbb{Z}_n$ times $\mathbf{e}\in \{-1,0,1\}^d$. Two queens are in conflict if they are mutually a move apart.
\end{definition}

\begin{definition}
The maximal number of queens that can be placed on $\mathbb{Z}_n^d$ without some two in conflict is $Q_d(n)$.
\end{definition}

Upon partitioning $\mathbb{Z}_n^{d+1}$ according to the value of the last coordinate, it is clear that $Q_{d+1}(n)\le nQ_d(n)$. Then, by induction, $Q_d(n)\le n^{d-1}$ and also
\begin{align*}
    Q_d(n)&\le n^{d-2}(n-1) \quad \text{if }2|n\text{ or }3|n, \\
    Q_d(n)&\le n^{d-2}(n-2) \quad \text{if }4|n\text{ or }3|n.
\end{align*}
Extrapolating this evidence, we could be led to believe the following.

\begin{conjecture}
    There is an absolute constant $c'>0$ such that for all $n$, either $Q_d(n)=n^{d-1}$ or $Q_d(n)\le n^{d-2}(n-c')$.
\end{conjecture}

Nudelman conjectured that for a figurine with moves perpendicular to a queen move, $n$ figurines can be placed on $\mathbb{Z}_n^d$ if and only if $(n,(2^d-1)!)=1$. Similarly:

\begin{conjecture}
    $Q_d(n)=n^{d-1}$ if and only if $(n,(2^d-1)!)=1$.
\end{conjecture}

For $n$ that have a prime divisor less than $2^d$, it is unknown whether this holds. However, the following result implies it is true whenever $n$ is a squarefree number.

\begin{theorem}\label{maine}
Let $n$ be of the form $n=\ell m$, where $1<\ell< 2^d$ and $\ell$ does not divide $m^{d-1}$. Then $Q_d(n)\le n^{d-1}-1$.
\end{theorem}

\begin{proof}
Suppose that $Q_d(n)=n^{d-1}$ for some $n$ of this form. Consider the function $f:\mathbb{Z}_\ell^d\to \mathbb{C}$, where $f(\mathbf{x})$ counts how many queens occupy the fields of $\mathbb{Z}_n^d$ that reduce to $\mathbf{x}\in\mathbb{Z}_\ell^d$ (componentwise) modulo $\ell$. Then for all $\mathbf{e}\in\{-1,0,1\}^d\setminus\{\mathbf{0}\}$,
$$\sum_{t\in \mathbb{Z}_\ell} f(\mathbf{x}+t\mathbf{e})=m^{d-1}.\eqno{(\star)}$$
To see this, let an $\mathbf{e}$-diagonal be an equivalence class of $n$ fields that differ by a multiple of $\mathbf{e}$. There is at most one queen in an $\mathbf{e}$-diagonal, yet $n^{d-1}$ queens in total, so exactly one queen. The sum counts the number of queens on a set of $\ell m^d=nm^{d-1}$ fields that can be partitioned into $\mathbf{e}$-diagonals, so it equals $m^{d-1}$.

If $\xi,\mathbf{x}\in \mathbb{Z}_\ell ^d$, let $\xi\cdot \mathbf{x}=\sum_{i=1}^d \xi_ix_i$. Let $\varepsilon=\exp\left(\frac{2\pi i}{\ell}\right)$ and $\hat{f}:\mathbb{Z}_\ell^d\to \mathbb{C}$,
$$\hat{f}(\xi)=\sum_{\mathbf{x}\in \mathbb{Z}_\ell^d}f(\mathbf{x})\varepsilon^{-\xi\cdot \mathbf{x}}$$
be the discrete Fourier transform (DFT) of $f$. The DFT of $\mathbf{x}\mapsto f(\mathbf{x}+t\mathbf{e})$ is then
$$\sum_{\mathbf{x}\in \mathbb{Z}_\ell^d}f(\mathbf{x}+t\mathbf{e})\varepsilon^{-\xi\cdot \mathbf{x}}=\sum_{\mathbf{y}\in\mathbb{Z}_\ell^d}f(\mathbf{y})\varepsilon^{-\xi\cdot (\mathbf{y}-t\mathbf{e})}=\hat{f}(\xi)\varepsilon^{t\xi\cdot \mathbf{e}}.$$
Therefore, applying the DFT to both sides of $(\star)$, by geometric sums, for all $\xi\neq \mathbf{0}$,
$$\hat{f}(\xi)\sum_{t\in \mathbb{Z}_\ell}(\varepsilon^{\xi\cdot \mathbf{e}})^t=0.$$
Now, there are $2^d$ sums of the form $\sum_{i\in I}\xi_i$, $I\subseteq\{1,\ldots,d\}$. By the pigeonhole principle, some two such subset sums modulo $\ell$ are equal. Subtracting the indicator vectors of these two subsets yields a vector $\mathbf{e}\in\{-1,0,1\}^d\setminus\{\mathbf{0}\}$ for which $\xi\cdot \mathbf{e}=0$. Using this $\mathbf{e}$, we deduce that $\ell \hat{f}(\xi)=0$, so $\hat{f}(\xi)=0$, if $\xi\neq\mathbf{0}$.

By Fourier Inversion,
$$\sum_{\xi \in \mathbb{Z}_\ell^d}\hat{f}(\xi)\varepsilon^{\mathbf{x}\cdot \xi}=\ell^df(\mathbf{x}),$$
we deduce that $f$ is a constant function. Thus, by $(\star)$, $\ell$ divides $m^{d-1}$. Absurd!
\end{proof}

Theorem \ref{maine} applies unless whenever $\ell|n$ and $\ell<2^d$, also $\ell|(n/\ell)^{d-1}$, i.e. $\ell^d|n^{d-1}$. For a prime $p$, let $\nu_p(N)$ denote the exponent of $p$ in the unique prime factorisation of the positive integer $N$. Under this notation, said condition is equivalent to, for all primes $p<2^d$, $\nu_p(n)=0$ or $\nu_p(n)\ge \frac{d}{d-1}\lfloor \log_p(2^d-1)\rfloor$. 

\section{Lower bounds in higher dimensions}

We prove the following result, already known for $d=3$ \cite{Kla}, \cite{AMY} and supposedly for $d>3$ \cite{VaR} in \cite{BeS}, using the base $2$ number system.

\begin{theorem}
Let $n$ be a positive integer that has no prime divisor less than $2^d$. Then $n^{d-1}$ many queens can be placed on $\mathbb{Z}_n^d$ without conflict -- as many as rooks.
\end{theorem}

\begin{proof}
We place queens on the fields
$$\{(t_1,-2t_1+t_2,-2t_2+t_3,\dots,-2t_{d-1}):t_1,\dots,t_{d-1}\in \mathbb{Z}_n\}.$$
If two queens are in conflict, then their displacement has the form
$$(t_1,-2t_1+t_2,-2t_2+t_3,\dots,-2t_{d-1})=(\epsilon_1 t_d,\epsilon_2 t_d,\epsilon_3 t_d,\dots,\epsilon_d t_d),$$
where $(\epsilon_1,\dots,\epsilon_d)\in \{-1,0,1\}^d$ is nonzero and $t_d\neq 0$. Solving, we see
$$t_1=\epsilon_1 t_d, \quad t_2=(\epsilon_2+2\epsilon_1)t_d,\quad t_3=(\epsilon_3+2\epsilon_{2}+4\epsilon_1)t_d,\quad \dots$$
which culminates with
$$0=(\epsilon_d+2\epsilon_{d-1}+\dots+2^{d-1} \epsilon_1)t_d.$$
The bracket contains a nonzero integer of at most $d$ digits in binary, of size less than $2^d$. Provided that $n$ is coprime to it, $t_d=0$ in $\mathbb{Z}_n$, and so the conflict is absurd.
\end{proof}

On the other hand, using the author's ideas, an LLM provided the following resolution to her conjecture.

\begin{theorem} \label{main}
Let $n,d$ be positive integers. There is a constant $c$, depending on $d$, such that one can always place $n^{d-1}-cn^{d-2}$ queens on $\mathbb{Z}_n^d$ without conflict.
\end{theorem}

\begin{proof}
Instead of the base $2$, we consider the base $b$ given by the product of all prime numbers less than $\max\{3^d/2,d!\}$. For each $\mathbf{e}\in\{-1,0,1\}^d\setminus\{\mathbf{0}\}$, let
$$m=m(\mathbf{e})=\mathbf{b}\cdot \mathbf{e}=e_1+be_2+\ldots+b^{d-1}e_d,$$
where $\mathbf{b}=(1,b,\ldots,b^{d-1})$. Since $\mathbf{e}$ and $-\mathbf{e}$ determine the same line, we may hereafter suppose that $m>0$. Let $L$ be the least common multiple of these $m(\mathbf{e})$.

\medskip

\underline{Claim.} For each $p|L$ and $1\le k\le \nu_p(L)$, there is a vector $\mathbf{v}=\mathbf{v}_{p,k}$ such that:
\begin{itemize}
    \item $\mathbf{v}\cdot \mathbf{e}=0$ if $\nu_p(m(\mathbf{e}))>k$,
    \item $\mathbf{v}\cdot \mathbf{e}\not\equiv 0\pmod p$ if $\nu_p(m(\mathbf{e}))=k$.
\end{itemize}

\begin{proof}[Proof of Claim]
If $p|b$, we can let $\mathbf{v}_{p,k}$ be the $k$-th unit vector, so that $\mathbf{v}_{p,k}\cdot \mathbf{e}=e_k$. This works because $\nu_p(m(\mathbf{e}))$ is the least index $i$ for which $e_i\neq 0$.

Otherwise, let us consider the vector subspace $V$ of $\mathbf{u}\in \mathbb{F}_p^d$ for which $\mathbf{u}\cdot \mathbf{e}=0$ (in $\mathbb{F}_p$!) whenever $\nu_p(m(\mathbf{e}))>k$. 

Suppose there is an $\mathbf{e}$ with $\nu_p(m(\mathbf{e}))=k$ in the orthogonal complement of $V$, which is the span of $\mathbf{e}'$ with $p^{k+1}|\mathbf{b}\cdot \mathbf{e}'$. As $p^{k+1}\not|\mathbf{b}\cdot \mathbf{e}=m(\mathbf{e})$, this is a contradiction. Therefore, for each $\mathbf{e}$ with $\nu_p(m(\mathbf{e}))=k$, $\mathbf{u}\cdot \mathbf{e}=0$ determines a proper subspace of $V$. 

Since there are less than $3^d/2$ such $\mathbf{e}$, the union of these subspaces contains less than $\frac{3^d}{2p}\le 1$ proportion of $V$ (for $p\ge 3^d/2$). Thus, there is a $\overline{\mathbf{v}}\in V$ not covered by them, such that in $\mathbb{F}_p$,
\begin{itemize}
    \item $\overline{\mathbf{v}}\cdot \mathbf{e}=0$ if $\nu_p(m(\mathbf{e}))>k$,
    \item $\overline{\mathbf{v}}\cdot \mathbf{e}\neq 0$ if $\nu_p(m(\mathbf{e}))=k$.
\end{itemize}
Hence, if componentwise, $\mathbf{v}_0\in\mathbb{Z}^d$ projects modulo $p$ to $\overline{\mathbf{v}}$, then
\begin{itemize}
    \item $\mathbf{v}\cdot \mathbf{e}=pa_{\mathbf{e}}$ for some $a_\mathbf{e}\in \mathbb{Z}$ if $\nu_p(m(\mathbf{e}))>k$,
    \item $\mathbf{v}_0\cdot \mathbf{e}\not\equiv 0\pmod p$ if $\nu_p(m(\mathbf{e}))=k$.
\end{itemize}

If we solve the $\mathbb{Q}$-linear system of equations $\mathbf{v}_1\cdot \mathbf{e}=a_\mathbf{e}$ (for the $\mathbf{e}$ with $\nu_p(m)>k$) by only adding or subtracting rows, the minors of the row matrix are invariant. (We iterate finding a non-vanishing column, running the Euclidean algorithm on it to add or subtract the respective rows, then setting aside the row containing the one remaining nonzero coefficient in the column.) By the triangle inequality, these minors are bounded by $d!$. Hence, knowing that a solution exists, there is a solution $\mathbf{v}_1$ obtained by adding, multiplying, and dividing by numbers not divisible by $p$. $\mathbf{v}=\mathbf{v}_0+p\mathbf{v}_1$ therefore satisfies both conditions, and if we multiply by $(d!)!^{d-1}$, it is certainly an integer vector that satisfies both conditions.
\end{proof}

\medskip

We need these vectors $\mathbf{v}_{p,k}$ to define the functions $A_{p,k}:(\mathbb{R}/\mathbb{Z})^d\to \mathbb{Z}$ by
$$A_{p,k}(\mathbf{x})=\left\lfloor p^k Lp^{-\nu_p(L)} \mathbf{v}_{p,k}\cdot \mathbf{x}\right\rfloor$$
aggregated in a function $a:(\mathbb{R}/\mathbb{Z})^d\to \{0,1,\ldots,L-1\}$ such that
$$a(\mathbf{x})\equiv 
\begin{cases} 
\sum_{k=1}^{\nu_p(L)} A_{p,k}(\mathbf{x}) \pmod{p^{\nu_p(L)}} \text{ if }p|n \\ 
0 \pmod{p^{\nu_p(L)}} \quad \text{ if }p\not |n
\end{cases}$$
as per the Chinese Remainder Theorem. Let $\overline{a}$ be $a$ reduced modulo $n$. 

Now, let us begin with $n^{d-1}$ queens, placed on $\mathbf{y}\in \mathbb{Z}_n^d$ for which
$$\mathbf{b}\cdot \mathbf{y}=\overline{a}(\mathbf{y}/n).$$
Note that for each possible value of $(y_2,\ldots,y_d)$, there is a unique such $\mathbf{y}$. Next, suppose the queens at $\mathbf{y}$ and $\mathbf{z}$ conflict: $\mathbf{z}-\mathbf{y}=t\mathbf{e}$ for some $\mathbf{e}\in\{-1,0,1\}^d\setminus\{\mathbf{0}\}$ and $t\in\mathbb{Z}_n$. Then by definition,
$$tm(\mathbf{e})=t\mathbf{b}\cdot \mathbf{e}=\mathbf{b}\cdot t\mathbf{e}=\mathbf{b}\cdot (\mathbf{z}-\mathbf{y})=\mathbf{b}\cdot \mathbf{z}-\mathbf{b}\cdot \mathbf{y}=\overline{a}(\mathbf{z}/n)-\overline{a}(\mathbf{y}/n).$$
Hence, for some integer $\theta$ that is $t$ modulo $n$, there is a $\gamma\in \mathbb{Z}$ such that
$$\theta m=\delta+n\gamma,$$
where $\delta:=a(\mathbf{z}/n)-a(\mathbf{y}/n)$ is bounded by $L-1$. We deduce $\frac{\theta}{n}=\frac{\delta}{mn}+\frac{\gamma}{m}$.

Define $\mathbf{w}=\mathbf{z}/n-\frac{\gamma}{m}\mathbf{e}\in (\mathbb{R}/\mathbb{Z})^d$. (Here, to be semantically precise, $\mathbf{z}/n$ is the well-defined projection of a preimage of $\mathbf{z}$ in $\mathbb{Z}^d$ scaled down by $n$, while $\mathbf{e}/m$ is the projection of a vector in $\mathbb{R}^d$.) This way, in $(\mathbb{R}/\mathbb{Z})^d$, $\mathbf{w}-\mathbf{y}/n=\frac{\delta}{mn}\mathbf{e}$. Hence,
$$A_{p,k}(\mathbf{w})=A_{p,k}(\mathbf{y}/n)$$
unless $A_{p,k}$ does not vanish and $\mathbf{v}_{p,k}\cdot (\mathbf{y}/n)$ is within distance $\left|\frac{\delta}{mn}\mathbf{v}_{p,k}\cdot \mathbf{e}\right|=\mathcal{O}(n^{-1})$ of a jump, some integer over $p^kLp^{-\nu_p(L)}$.

Here, to each queen $\mathbf{y}$, there corresponds a unique $\mathbf{y}_0$ with $\mathbf{b}\cdot \mathbf{y}_0=0$ that differs from $\mathbf{y}$ but in its first coordinate and by at most $L-1$. Over these $\mathbf{y}_0$, each value of $\mathbf{v}_{p,k}\cdot \mathbf{y}_0\in \mathbb{Z}_n$ is attained $\mathcal{O}(1)$ times. The reason for this is that $\mathbf{v}_{p,k}$ is not a multiple of $\mathbf{b}$, because $p\not| \mathbf{v}_{p,k}\cdot \mathbf{e}'$ and $p^k|\mathbf{b}\cdot \mathbf{e}'$, if there is a direction $\mathbf{e}'$ with $m(\mathbf{e}')=k$ (or else we can specify $\mathbf{v}_{p,k}=\mathbf{0}$ to make $A_{p,k}$ vanish). For $\mathbf{v}_{p,k}\cdot (\mathbf{y}_0/n)$ to be not within $\mathcal{O}(n^{-1})+\frac{L-1}{n}$ of $\mathcal{O}(1)$ discontinuities, for $\mathcal{O}(1)$ possible $p$ and $k$, it is a proportion $\mathcal{O}(n^{-1})$ of the $\mathbf{y}_0$ that must be discarded. Therefore, upon discarding $\mathcal{O}(n^{d-2})$ of the queens, we can be assured, at this stage, that 
$$a(\mathbf{w})=a(\mathbf{y}/n).$$

On the other hand, $m$ divides
$$\theta m=\delta+n\gamma=a(\mathbf{z}/n)-a(\mathbf{y}/n)+n\gamma$$
and hence $a(\mathbf{z}/n)-a\left(\mathbf{z}/n-\frac{\gamma}{m}\mathbf{e}\right)+n\gamma$. Fix a prime $p|L$. Let $j=\nu_p(m)$. We claim $p^j|\gamma$. If $p\not|n$, each value of $a$ is divisible by $p^j$, so indeed, $p^j|\gamma$. If $p|n$,
$$p^j|n\gamma+\sum_{k=1}^{\nu_p(L)} \left[A_{p,k}(\mathbf{z}/n)-A_{p,k}\left(\mathbf{z}/n-\frac{\gamma}{m}\mathbf{e}\right)\right].$$
Consider that $\mathbf{v}_{p,k}\cdot \mathbf{e}=0$ for $k<\nu_p(m)$ by the Claim, in which case we obtain $A_{p,k}(\mathbf{z}/n)=A_{p,k}\left(\mathbf{z}/n-\frac{\gamma}{m}\mathbf{e}\right)$. Yet for $k\ge \nu_p(m)$, $p^kLp^{-\nu_p(L)}\mathbf{v}_{p,k}\cdot\frac{\gamma}{m}\mathbf{e}\in\mathbb{Z}$, because $m|L$ and $\nu_p(m)\le k$. It follows that
$$p^j|n\gamma+\sum_{k=j}^{\nu_p(L)}p^kLp^{-\nu_p(L)}\mathbf{v}_{p,k}\cdot\frac{\gamma}{m}\mathbf{e}=\left(n+\sum_{k=j}^{\nu_p(L)}p^{k-j}\frac{Lp^{-\nu_p(L)}}{mp^{-j}}(\mathbf{v}_{p,k}\cdot\mathbf{e})\right)\gamma.$$
Crucially, $\mathbf{v}_{p,j}\cdot \mathbf{e}$ is nonzero modulo $p$, because $j=\nu_p(m)$. Thus, the bracket is, too. We deduce $p^j|\gamma$. Since $p|L$ was arbitrary, it follows that $m|\gamma$.

This implies $\frac{\gamma}{m}\in\mathbb{Z}$, so that $\mathbf{w}=\mathbf{z}/n$. Hence, $a(\mathbf{z}/n)=a(\mathbf{y}/n)$. But then $\delta=0$, so $m|\gamma$ implies $n|\theta$, i.e. $t=0$.

With these $n^{d-1}-\mathcal{O}(n^{d-2})$ queens, no two distinct queens are in conflict.
\end{proof}

\section{Acknowledgements and Declaration}

The author is grateful to Imre Leader for suggesting this topic of recreational mathematics and to Trinity College, Cambridge for PhD funding. She would like to thank Candida Bowtell for kindly pointing out a number of previous results \cite{Kla}, \cite{VaR}, \cite{BeS}, \cite{Nud}, including Monsky's characterisation of $Q_2(n)$ \cite{Mon}. The proofs of Theorems \ref{maine} and \ref{main} were generated by an LLM, then communicated by the author.

\textit{E-mail address:} \texttt{kkw25@cantab.ac.uk}

\end{document}